\documentclass[12pt]{amsart}
\usepackage[T1]{fontenc}
\usepackage[utf8]{inputenc}
\usepackage{amssymb}
\usepackage{amsmath}
\usepackage{amsfonts}
\usepackage{amscd}
\usepackage{verbatim}
\usepackage[mathscr]{euscript}
\usepackage[dvipsnames,usenames]{color}

\numberwithin{equation}{section}

\newcommand{\n}[1]{\|#1\|}
\def\a{\alpha}       \def\b{\beta}           \def\g{\gamma}
            \def\ev{\varepsilon}
\def\z{\zeta}                  \def\th{\theta}
          \def\la{\lambda}     
                    \def\p{\psi}

\def\t{\tau}          \def\m{\mu}

           \def\O{\Omega}

\def\dth{\frac{d\theta}{2\pi}}
\newcommand{\hol}{{\mathcal H}}

\DeclareMathOperator{\og}{O}

\def\D{{\mathbb D}}
\def\T{{\mathbb T}}
\def\C{{\mathbb C}}  
 \def\R{{\mathbb R}}

 \def\cb{{\mathcal B}}

\newtheorem{theorem}{Theorem}[section]
\newtheorem{corollary}[theorem]{Corollary}

\begin{document}

\title[Semigroups in Morrey spaces]{Semigroups of composition
operators in analytic Morrey spaces}

\author[P. Galanopoulos]{P. Galanopoulos}
\address{Department of Mathematics,
Aristotle University of Thessaloniki, 54124 Thessaloniki, Greece}
\email{petrosgala@math.auth.gr}

\author[N. Merch\'{a}n]{N. Merch\'{a}n}
\address{Departamento de Matem\'{a}tica Aplicada, ETSII, 
Universidad de M\'{a}laga, Campus de Teatinos, 29071 M\'{a}laga, Spain} \email{noel@uma.es}

 \author[A. G. Siskakis]{A. G. Siskakis}
 \address{Department of Mathematics,
Aristotle University of Thessaloniki, 54124 Thessaloniki, Greece}
 \email{siskakis@math.auth.gr}

\date{\today}

\begin{abstract}
Analytic Morrey spaces  belong to the class of function spaces which, like BMOA, are defined in terms of the degree of
oscillation on the boundary of functions analytic in the  unit disc. We consider semigroups of composition operators on these
spaces and focus on the question of strong continuity. It is shown that these semigroups behave like on BMOA.
\end{abstract}
\maketitle

\section{Introduction}

Let $\D$ be the unit disc in the complex plane $\C$. A family $(\phi_t)_{t\geq 0}$ of
analytic self maps of $\D$ is a \textit{semigroup of  functions}
if  the following conditions hold\\

\noindent
(s1) $\phi_0$ is the identity map of $\D$,\\
(s2) $\phi_s\circ \phi_t=\phi_{s+t} $ for all $s, t \geq 0$,\\
(s3) $\phi_t\to \phi_0$ uniformly on compact subsets of $\D$ as $t\to 0$.\\

Each such $(\phi_t)$ induces a semigroup $(C_t)$ of linear transformations on the Fr\`{e}chet  space $\hol(\D)$ of all analytic
functions on $\D$, by composition
$$
C_t(f)(z)=f(\phi_t(z)), \quad f\in \hol(\D),
$$
and the following pointwise convergence holds
$$
\lim_{t\to 0^{+}}f(\phi_t(z))= f(z), \quad z\in \D.
$$
If $(X, \n{\,}_X)$ is a Banach space consisting of analytic functions on $\D$, and
$(\phi_t)$ is a semigroup of functions  we say that $(\phi_t)$ \textit{acts} on $X$ if
each composition operator $C_t$ is  bounded  on $X$. If in addition
\begin{equation}\label{sc}
 \lim_{t\to 0^{+}}\n{f\circ\phi_t-f}_X =0
\end{equation}
for each $f\in X$ then  $(C_t)$ is  \textit{strongly continuous}  on $X$. In this case
it is interesting to study  how   operator theoretic properties of the operator
semigroup $(C_t)$ relate to   function theoretic properties of $(\phi_t)$.

The study of composition semigroups on spaces of analytic functions started in \cite{Berkson-Porta}
where  E. Berkson and H.
Porta  studied  the basic properties of semigroups of functions $(\phi_t)$, and proved  that they
induce  strongly continuous
 operator semigroups   on Hardy spaces. Other  authors  studied   strong continuity on Bergman
spaces $A^p$, Dirichlet spaces,   the spaces $BMOA$ and  the Bloch space $\cb$, and their
subspaces $VMOA$ and $\cb_0$.
More recently strong continuity was studied on  mixed norm spaces $H(p, q, \alpha)$
in \cite{ArevContrerasRondrig}. Also  in
the recent article \cite{AndersonJovovicSmith} the authors gave a unified proof of
several earlier results, and showed that  no
nontrivial composition semigroup is strongly continuous on BMOA, thus extending similar results known
for the space of bounded analytic functions $H^{\infty}$ and the Bloch space $\cb$.

The purpose of this article is to study strong continuity of composition semigroups on
analytic Morrey spaces
$H^{2, \la}$. These spaces are defined in terms of the  oscillation of the  boundary
function. For $0<\la<1$ and for arcs  $I$ on the circle of length $|I|$, the oscillation is compared to
$|I|^{\la}$. The end point $\la=1$ corresponds to     BMOA.
For  each $0<\la<1$,  $BMOA\subset H^{2, \la}\subset H^2$,
 and  these spaces share several properties of BMOA. Their  definition and their properties are
 given in the next section.

For every analytic self map $\phi$ of the disc, the composition operator $f\to f\circ\phi$ is a bounded operator on
each $H^{2, \la}$. We will study   strong continuity of composition semigroups $(C_t)$ and will prove that
$(C_t)$ is not strongly continuous on $H^{2, \la}$ unless the it is trivial.

We will denote constants in various inequalities below by $C, C', ...$ and  their values may
change from one step to the next.

\vspace{1cm}

\section{Background on semigroups and Morrey spaces}

\subsection{Semigroups of functions and composition operators}
If $(\phi_t)$ is a semigroup of functions then each $\phi_t$ is univalent.
The limit
$$
G(z)=\lim_{t\to 0^{+}}\frac{\phi_{t}(z)-z}{t}, \quad z\in\D.
$$
exists  uniformly on compact subsets of $\D$, it is therefore an analytic function on
$\D$,  and is called the infinitesimal generator of $(\phi_t)$. The trivial semigroup,
$\phi_t(z)\equiv z$ for all $t$,  corresponds to the generator $G\equiv 0$. This
trivial case will be ignored.  The functional equation
\begin{equation}\label{gener}
G(\phi_{t}(z))=\frac{\partial\phi_{t}(z)}{\partial t}=G(z)\frac
{\partial\phi_{t}(z)}{\partial z},\quad  z\in\D, \,\, t\geq 0,
\end{equation}
is satisfied, and the generator $G$  has  a unique representation
$$
G(z)=(\overline{b}z-1)(z-b)P(z),\quad z\in\D
$$
where $b\in\overline{\D}$ and $P\in\hol(\D)$ with $\text{Re}(P)\geq 0$ on  $\D$.
The pair  $(b,P)$ is  uniquely determined by $\left(\phi_{t}\right)$. Conversely
every pair $(b, P)$ with $b\in\overline{\D}$ and $P\in \hol(\D)$ with
$\text{Re}(P)\geq 0$ gives rise to a semigroup of functions. The  point $b$ is  the
\textit{Denjoy-Wolff point} of the semigroup. When  $b\in \D$ then $b$ is the
common fixed point of all  $\phi_t$. The semigroup $(\phi_t)$ may  be normalized
by composing with automorphisms of the disc, and assume that $b=0$, and it can be represented
as
$$
\phi_t(z)=h^{-1}(e^{-ct}h(z)),\quad z\in \D, t\geq 0
$$
where $h$ is a univalent function mapping $\D$ onto a spirallike domain $\Omega$, with $h(0)=0$
and Re$(c)\geq 0$.

 When  $b\in \partial\D$ then every  orbit  $\{\g_z(t)=\phi_t(z): t\geq 0\}$
 converge to $b$ as $t \to \infty$ for each  $z\in \D$. In this case $b$ is also
a common fixed point in the sense that $\lim_{r\to 1^{-}}\phi_t(rb)=b$ for all $t$.
By a normalization it can be assumed that $b=1$, and $(\phi_t)$ has a representation
$$
\phi_t(z)=h^{-1}(h(z)+ct),\quad z\in \D, t\geq 0
$$
where $h$ is a univalent function mapping $\D$ onto a close-to-convex domain $\Omega$ with $h(0)=0$ and
Re$(c)>0$.
The above and other details can be found in \cite{Berkson-Porta} and
\cite{Siskakis2}.

If   $(\phi_t)$ is a semigroup and   $X$  a Banach space of analytic functions on $\D$
such that $(\phi_t)$ acts on $X$ but the induced operator semigroup $(C_t)$ is not
strongly continuous on $X$, it may still happen that $(C_t)$ is strongly continuous
on a nontrivial proper subspace of $X$. Thus  we define
 $$
 [\phi_t, X]=\{f\in X: \n{f\circ\phi_t - f}_X\to 0 \,\, \text{as}\,\, t\to 0\}.
 $$
 A triangle inequality shows that if $\sup_{0\leq t\leq 1}\n{C_t}_X<\infty$ then
 $[\phi_t, X]$ is a closed subspace of $X$ which is maximal with respect to the strong continuity
 requirement. If in addition $X$ contains the constant functions then
 \begin{equation}\label{max}
 [\phi_t, X]=\overline{\{f\in X: Gf'\in X\}}
 \end{equation}
 where $G$ is the generator of $(\phi_t)$, see \cite{BCDMPS}.
Note that in this notation, $[\phi_t, X]=X$ means that $(\phi_t)$ induces a
semigroup
 of bounded composition operators, which is strongly continuous on $X$.
If the space $X$ contains $H^{\infty}$ and $X_0$ is the closure of polynomials in
$X$, then for any semigroup $(\phi_t)$ we have
$$
X_0\subseteq [\phi_t, X],
$$
\cite[Corollary 1.3]{AndersonJovovicSmith}. This in particular implies that if
$H^{\infty}\subset X$ and polynomials are dense in $X$ then for every  $(\phi_t)$
the induced semigroup $(C_t)$ is strongly continuous on $X$.  This is the case for
Hardy spaces $H^p, 1\leq p<\infty$, \cite{Berkson-Porta} and the Bergman spaces
$A^p, 1\leq p<\infty, $ \cite{Siskakis0}. The above  also applies when $X=BMOA$
and implies that for every $(\phi_t)$ the induced  $(C_t)$ is strongly continuous  on $BMOA_0=VMOA$.

D. Sarason \cite{Sarason1} proved  that  VMOA consists of  those $f$ in BMOA that can be
approximated by their rotations inside BMOA. This translates in our language to that for the rotation
semigroup  $\phi_t(z) = e^{it}z$ we have   $[\phi_t, BMOA]=VMOA$. The same is true for the dilation semigroup
$\phi_t(z)= e^{-t}z$, and there other  semigroups $(\phi_t)$, different from
 rotations and dilations, whose maximal subspace of strong continuity coincides with VMOA;  such
 semigroups were studied in \cite{BCDMS}. There are analogous results for the Bloch space \cite{BCDMPS}.

Additional information for strong continuity of $(C_t)$  can be found in \cite{Siskakis1} for the
Dirichlet space, in \cite{Galano} for
weighted Dirichlet spaces, in \cite{Wirths-Xiao} for $Q_p$ spaces, in
\cite{AndersonJovovicSmith} and  \cite{Contreras-Diaz} for the disc algebra, and in
\cite{ArevContrerasRondrig} for mixed norm spaces.

\subsection{Analytic Morrey spaces}  Let  $L^2(\T)$ denote the  Hilbert space
of square integrable functions on the unit circle $\T$, and  $H^2$ the Hardy space of
all analytic functions $f$ on $\D$ such that
$$
\n{f}_{H^2}^2= \sup_{r<1}\int_0^{2\pi}|f(re^{i\theta})|^2\frac{d\theta}{2\pi}<\infty.
$$
For every $f\in H^2$ the radial limits $\tilde{f}(e^{i\theta})=\lim_{r\to 1^{-}}f(re^{i\theta})$
 exist a.e. on $\T$ and the boundary function $\tilde{f}$ is in $L^2(\T)$ with
 $\n{\tilde{f}}_{L^2(\T)}=\n{f}_{H^2}$. In addition the map $f\to \tilde{f}$ is an injection,
 and identifies $H^2$ as a closed subspace of $L^2(\T)$.
For an arc $I\subset\T$ of normalized length  $|I|=\frac{1}{2\pi}\int_{I}d\th$
and for $f\in L^2(\T)$
 let
$$
f_I=\frac{1}{|I|}\int_{I}f(e^{i\th})\,\dth
$$
 be  the average of $f$ over $I$. The quantity
\begin{equation}\label{1}
I(f)=\int_I|f(e^{i\th})-f_I|^2\,\dth= \int_I|f(e^{i\th})|^2\,\dth
-|I||f_I|^2
\end{equation}
converges  to $ 0$ as  $|I|\to 0$, and the  rate of the convergence depends on the degree of oscillation
of $f$ around its average $f_I$. Given  $0\leq \la \leq 1$ we may  isolate those $f\in L^2(\T)$ for
which  $I(f)/|I|^{\la} $ stays bounded or tends to $0$  as $|I|\to 0$, and thus define the space
\begin{equation}\label{2}
L^{2,\la}= \left\{f\in L^2(\T): \sup_{I\subset \T}\frac{1}{|I|^{\la}}\int_I|f(e^{i\th})-f_I|^2\,\dth<\infty\right\},
\end{equation}
and its  subspace
\begin{equation}\label{2_0}
L^{2,\la}_0= \left\{f\in L^2(\T): \lim_{|I|\to 0}\frac{1}{|I|^{\la}}\int_I|f(e^{i\th})-f_I|^2\,\dth =0\right\}.
\end{equation}
These are  the Morrey   spaces. They were  introduced
in connection with regularity of solutions of partial differential equations,
see  \cite{Morrey}, \cite{Camp1}, \cite{Camp2}. They are linear spaces complete  under the seminorm
\begin{equation}\label{s1}
p_1(f)=
 \sup_{I\subset \T}\left(\frac{1}{|I|^{\la}}\int_I|f(e^{i\th})-f_I|^2\,d\th\right)^{1/2}.
\end{equation}
For $\la=1$ these spaces are
BMO and VMO, the spaces of bounded and vanishing mean oscillation  respectively.
For $\la=0$ both spaces coincide with  $L^2(\T)$.

 Using the identification of $ H^2$ as a closed subspace of $L^2(\T)$, the analytic  Morrey spaces are
 $$
 H^{2, \la}=L^{2,\la}\cap H^2, \qquad H^{2, \la}_0=L^{2,\la}_0\cap H^2,
 $$
that is,  they consist of those analytic functions in $H^2$ whose boundary values belong to $L^{2, \la}$,
respectively  to $L^{2, \la}_0$. The restriction of (\ref{s1}) is then a seminorm  on $H^{2, \la}$. A second
seminorm on $H^{2, \la}$ is
\begin{equation}\label{s2}
p_2(f)=
\sup_{a\in \D}(1-|a|^2)^{\frac{1-\la}{2}}\n{f\circ\phi_a -f(a)}_{H^2},
\end{equation}
where  $\phi_a(z)=\frac{a-z}{1-\bar{a}z}$, $ |a|<1$, are the Mobius automorphism of $\D$ and this
seminorm is  equivalent to $p_1$, see \cite{Xiao-Xu}. The space $H^{2, \la}_0$ contains those $f$ for which
$$
\lim_{|a|\to 1}(1-|a|^2)^{\frac{1-\la}{2}}\n{f\circ\phi_a -f(a)}_{H^2}=0.
$$
  A third  equivalent seminorm is
\begin{equation}\label{s3}
p_3(f)=  \sup_{I\subset\T}\left(\frac{1}{|I|^{\la}}\int_{S(I)}|f'(z)|^2 (1-|z|)\,dA(z)\right)^{1/2}
\end{equation}
where $S(I)=\{z\in \D: 1-|I|\leq |z|<1, z/|z|\in I\}$ is the  Carleson box based on
the arc $I\subset\T$ and $dA(z)$ is the normalized Lebesgue area measure, see
\cite{Liu-Lou1}. Functions in $H^{2, \la}_0$ are then characterized by the condition
$$
\lim_{|I|\to 0}\frac{1}{|I|^{\la}}\int_{S(I)}|f'(z)|^2 (1-|z|)\,dA(z)=0.
$$

Each of the above seminorms becomes a norm  on $H^{2,\, \la}$ by
adding $|f(0)|$ to it. We will write $\n{f}_{2,  \la}$ for each of the three equivalent
norms, thus ignoring the involved constants.

We list below some properties that are needed later.
For  $0<\la\leq 1$, $H^{2, \la}_0$ is  the closure of the polynomials in
$H^{2, \la}$, and for  $0< \la < \mu \leq  1$,
$$
H^2\supset H^{2, \la}\supset H^{2, \la}_0\supset  H^{2, \mu}\supset  H^{2, \mu}_0\supset  BMOA\supset VMOA.
$$
Further  $H^{2, \la}$  contains the Hardy space $ H^{ \frac{2}{1-\la}}$, \cite{Liu-Lou2} and for each $f\in H^{2, \la}$
\begin{equation}\label{growth}
|f(z)|\leq \frac{C\n{f}_{2, \la}}{(1-|z|)^{\frac{1-\la}{2}}}, \quad z\in \D.
\end{equation}
 The function $f(z)=(1-z)^{-\frac{1-\la}{2}}$ attains the maximum growth and belongs to $H^{2, \la}$.
If $\phi:\D\to \D$ is analytic then the composition operator $ C_{\phi}(f)=f\circ\phi$ is bounded on $H^{2, \la}$ and
$$
\n{C_{\phi}(f)}_{2,\la}\leq C\left(\frac{1+|\phi(0)|}{1-|\phi(0)|}\right)^{\frac{1-\la}{2}}\n{f}_{2, \la}
$$
More information and properties of Morrey spaces  can be found in \cite{Liu-Lou1}, \cite{Liu-Lou2},
 \cite{Wang-Xiao}, \cite{Xiao-Yuan}, \cite{Xiao-Xu}

\section{Semigroups in $H^{2, \la}$}

Suppose $0<\la<1$. Since $H^{\infty}\subset H^{2, \la}$ and $H^{2, \la}_0$ is
the closure of polynomials in $H^{2, \la}$, it follows from \cite[Corollary
1.3]{AndersonJovovicSmith} that for every semigroup $(\phi_t)$,
$$
H^{2, \la}_0\subseteq [\phi_t, H^{2, \la}] \subseteq H^{2, \la}.
$$
The question arises if there can be equality in either of the above containments. The
following theorem is the analogue of  Sarason's
characterization of VMOA functions by their property that they can be approximated inside BMOA by their rotations
or dilations.

\begin{theorem}
Suppose $0<\la<1$. For $f\in H^{2, \la}$ the following are equivalent\\
(1) $f\in H^{2, \la}_0$\\
(2) $\lim_{t\to 0^{+}}\n{f(e^{it}z)-f}_{2, \la}=0$\\
(3) $\lim_{t\to 0^{+}}\n{f(e^{-t}z)-f}_{2, \la}=0$
\end{theorem}
\begin{proof}
Let $I\subset \T$ be an arc. An easy computation gives
\begin{equation}\label{id}
|f|_I^2-|f_I|^2=
\frac{1}{|I|}\int_I |f(e^{i\theta})-f_I|^2\,d\theta=
\frac{1}{2|I|^2}\int_I\int_I|f(e^{i\theta})-f(e^{i\omega})|^2\,d\theta d\omega
\end{equation}
(see  \cite[Theorem 9.24]{Zhu}), and in particular we have
$$
\frac{1}{|I|^{\la}}\int_I
|f(e^{i\theta})-f_I|^2\,d\theta=|I|^{1-\la}(|f|_I^2-|f_I|^2).
$$
$(1)\Rightarrow (2)$. Suppose $f\in H^{2, \la}_0$ and write
$F_t(z)=f(e^{it}z)-f(z)$. We need to show
$$
\lim_{t\to 0^{+}}\n{F_t}_{2, \la}=
\lim_{t\to 0^{+}}\sup_{I}\left(\frac{1}{|I|^{\la}}
\int_I|F_t(e^{i\theta})-(F_t)_I|^2\,d\theta\right)^{1/2}=0.
$$
Since $f\in H^{2, \la}_0$ we have
$\lim_{|I|\to 0}|I|^{\frac{1-\la}{2}}(|f|_I^2-|f_I|^2)^{1/2}=0$
 so for a given $\ev>0$ there is
$\delta\in (0,1)$ such that if  $|I|<\delta$ then $
|I|^{\frac{1-\la}{2}}(|f|_I^2-|f_I|^2)^{1/2}<\frac{\ev}{2}. $ Thus for an arc $I$
and its rotations $I_t=e^{it}I$  with  $|I|=|I_t|<\delta$ we obtain
\begin{align*}
 \left(\frac{1}{|I|^{\la}}\int_I|F_t(e^{i\theta})-(F_t)_I|^2\,d\theta\right)^{1/2}
&=|I|^{\frac{1-\la}{2}}(|F_t|_I^2-|(F_t)_I|^2)^{1/2}\\
&\leq |I|^{\frac{1-\la}{2}}(|f|_I^2-|f_I|^2)^{1/2}+
|I|^{\frac{1-\la}{2}}(|f|_{I_t}^2-|f_{I_t}|^2)^{1/2}\\
&=|I|^{\frac{1-\la}{2}}(|f|_I^2-|f_I|^2)^{1/2}+
|I_t|^{\frac{1-\la}{2}}(|f|_{I_t}^2-|f_{I_t}|^2)^{1/2}\\
&\leq \frac{\ev}{2}+\frac{\ev}{2}\\
&=\ev
\end{align*}
for all $t\geq 0$. On the other hand if $|I|>\delta$ then
\begin{align*}
|I|^{\frac{1-\la}{2}}(|F_t|_I^2-|(F_t)_I|^2)^{1/2}&\leq |I|^{\frac{1-\la}{2}}(|F_t|^2_I)^{1/2}\\
&=|I|^{\frac{1-\la}{2}}\left(\frac{1}{|I|}
\int_I|f(e^{i(t+\theta)})-f(e^{i\theta})|^2
\,d\theta\right)^{1/2}\\
&\leq \frac{1}{\delta^{\la/2}}\left(\int_{\T}|f(e^{i(t+\theta)})-f(e^{i\theta})|^2
\,d\theta\right)^{1/2}
\end{align*}
By the continuity of the integral,  for $f\in L^2(\T)$ there is a $\t>0$ such that if
$0\leq t<\t$,
$$
\left(\int_{\T}|f(e^{i(t+\theta)})-f(e^{i\theta})|^2
\,d\theta\right)^{1/2}<\delta^{\la/2}\ev.
$$
It follows that for $0\leq t<\t$,
$$
\sup_{I}\left(\frac{1}{|I|^{\la}}
\int_I|F_t(e^{i\theta})-(F_t)_I|^2\,d\theta\right)^{1/2}<\ev,
$$
and this shows that (1) implies (2).\\
$(2)\Rightarrow (3)$. Suppose (2) holds for an $f\in H^{2, \la}$.  For $0<r<1$
write $f_r(z)=f(rz)$.  We will show equivalently that $\lim_{r\to
1^{-}}\n{f_r-f}_{2,\la}=0$. By the Poisson integral formula we have
$$
f(e^{i\theta})-f(re^{i\theta})=\frac{1}{2\pi}\int_0^{2\pi}P_r(t)[f(e^{i\theta})- f(e^{i(\theta-t)})]\,dt
$$
and an application of Fubini's theorem gives for every small positive $\delta$,
\begin{align*}
\n{f-f_r}_{2, \la}&\leq \frac{1}{2\pi}\int_0^{2\pi}P_r(t)\n{f-f(e^{-it}z )}_{2, \la}\,dt\\
&=\frac{1}{2\pi} \int_{|t|<\delta}P_r(t)\n{f(e^{it}z )-f}_{2, \la}\,dt+
\frac{1}{2\pi}\int_{\delta\leq |t|<\pi}P_r(t)\n{f(e^{it}z )-f}_{2, \la}\,dt\\
&\leq \frac{1}{2\pi}\int_{|t|<\delta}P_r(t)\n{f(e^{it}z )-f}_{2, \la}\,dt+
\frac{2\n{f}_{2, \la}}{2\pi}\int_{\delta\leq |t|<\pi}P_r(t)\,dt,
\end{align*}
where we have taken into account that $\n{f(e^{it}z )}_{2, \la}=\n{f}_{2, \la}$
for any $f\in H^{2, \la}$ and $t$ real.

If $\ev>0$ is given, using the properties of the Poisson kernel and the assumption
that (2) holds,  the first integral can be made smaller than $\ev/2$ by choosing
$\delta$ small. Fix such a  $\delta$, then  the integral in the second term tends to
$0$ as $r\to 1^{-}$, thus the second term is also less than $\ev/2$ for $r$
sufficiently close to $1$.
This  shows that (2) implies (3).\\
$(3) \Rightarrow (1)$. This implication is obvious since each $f_r$ is analytic on a
larger disc of radius $1/r$ hence $f_r\in H^{2, \la}_0$, and $H^{2, \la}_0$ is
closed in $H^{2, \la}$.
\end{proof}

The above theorem says that for $\phi_t(z)=e^{it}z$ and $\psi_t(z)=e^{-t}z$,
$$
[\phi_t, H^{2, \la}]=[\psi_t, H^{2, \la}]=H^{2, \la}_0.
$$
There are however semigroups $(\phi_t)$   such that the  inclusion $H^{2,
\la}_0\subset [\phi_t, H^{2, \la}]$ is proper.  For example for
$\phi_t(z)=e^{-t}z+1-e^{-t}$  the function
$$
f_{\la}(z)=\frac{1}{(1-z)^{\frac{1-\la}{2}}}
$$
belongs to $ H^{2, \la}\setminus H^{2, \la}_0$, and satisfies
$$
f_{\la}(\phi_t(z))=e^{t\frac{1-\la}{2}} f_{\la}(z),
$$
so
$$
\n{f_{\la}\circ\phi_t -f_{\la}}_{2, \la}= (e^{t\frac{1-\la}{2}}-1)\n{ f_{\la}}_{2,\la} \to 0 \,\, \text{as} \,\, t\to 0.
$$
Thus $f_{\la}\in [\phi_t, H^{2, \la}]$ and  $H^{2,\la}_0\subsetneq [\phi_t, H^{2, \la}]$. Other examples of semigroups
for which $[\phi_t, H^{2, \la}]$ is strictly larger than $ H^{2, \la}_0$ can be constructed
easily. For example let $h(z)=(\frac{1+z}{1-z})^{\frac{1-\la}{2}}-1$   and
$\phi_t(z)=h^{-1}(e^{-t}h(z))$.  Then
$$
\n{h\circ\phi_t-h}_{2, \la}= (1-e^{-t})\n{h}_{2, \la}\to 0 \,\, \text{as} \,\, t\to 0,
$$
so $h\in [\phi_t, H^{2, \la}]$ while $h\notin H^{2, \la}_0$.

The following theorem gives a sufficient condition on the generator $G$ of a
semigroup which implies that $[\phi_t, H^{2, \la}]=H^{2, \la}_0$.

\begin{theorem}
Let $(\phi_t)$ be a semigroup of functions with generator $G$ and $0<\la<1$.
Assume that
$$
\lim_{|I|\to 0}\frac{1}{|I|}\int_{S(I)} \frac{1-|z|}{|G(z)|^2}\,dA(z)=0.
$$
Then $[\phi_t, H^{2, \la}]=H^{2, \la}_0$.
\end{theorem}

\begin{proof}
Since $[\phi_t, H^{2, \la}]=\overline{\{f\in H^{2, \la}: Gf'\in H^{2, \la}\}}$ it
suffices to show that
$$
\{f\in H^{2, \la}: Gf'\in H^{2, \la}\}\subset H^{2, \la}_0
$$
Let $g\in H^{2, \la}$ such that $Gg'\in H^{2, \la}$. We will show that $g\in H^{2,\la}_0$ by showing that
$$
\lim_{|I|\to 0}\frac{1}{|I|^{\la}}\int_{S(I)} |g'(z)|^2(1-|z|)\,dA(z)=0
$$
For an arc $I$ on the circle with center $e^{i\th}$ let $a_I=(1-|I|)e^{i\th}\in\D$.
Writing $F(z)=G(z)g'(z)$ we have
\begin{align*}
\frac{1}{|I|^{\la}}&\int_{S(I)} |g'(z)|^2(1-|z|)\,dA(z)=\frac{1}{|I|^{\la}}\int_{S(I)} |G(z)g'(z)|^2\frac{1-|z|}{|G(z)|^2}\,dA(z)\\
&=\frac{1}{|I|^{\la}}\int_{S(I)} |F(z)|^2\frac{1-|z|}{|G(z)|^2}\,dA(z)\\
&\leq\frac{2}{|I|^{\la}}\int_{S(I)} |F(z)-F(a_I)|^2\frac{1-|z|}{|G(z)|^2}\,dA(z)+
\frac{2}{|I|^{\la}}\int_{S(I)} |F(a_I)|^2\frac{1-|z|}{|G(z)|^2}\,dA(z)\\
&\leq \frac{2}{|I|^{\la}}\int_{S(I)} |F(z)-F(a_I)|^2\frac{1-|z|}{|G(z)|^2}\,dA(z)
+ \frac{C\n{F}_{2,\la}^2}{(1-|a_I|)^{1-\la}}\frac{2}{|I|^{\la}}\int_{S(I)} \frac{1-|z|}{|G(z)|^2}\,dA(z)\\
&\leq \frac{2}{|I|^{\la}}\int_{S(I)} |F(z)-F(a_I)|^2\frac{1-|z|}{|G(z)|^2}\,dA(z)+C'\n{F}_{2,\la}^2
\frac{1}{|I|}\int_{S(I)} \frac{1-|z|}{|G(z)|^2}\,dA(z)\\
&=A_I+B_I,
\end{align*}
where we have used the growth inequality for $F\in H^{2,\la}$ and the fact that
$1-|a_I|=|I|$. By the hypothesis the second term $B_I$ tends to $0$ as $|I|\to 0$.
To conclude the proof it suffices to prove that  $A_I\to 0$ as $|I|\to 0$.

To prove that $\lim_{|I|\to 0}A_I=0$ recall first that there is an absolute constant
$C$ such that if  $I\subset \T$ is an arc then
$$
\frac{1-|a_I|}{|1-\overline{a_I}z|^2}\geq\frac{C}{|I|}
$$
for all $z\in S(I)$.  Next fixing for the moment an arc  $I$, let $\m$ be the measure
on $\D$ defined by
$$
\m(E)=\m_I(E)=\int_{S(I)\cap E}\frac{1-|z|}{|G(z)|^2}\,dA(z)
$$
for each Borel subset $E$ of $\D$. We will suppress the index $I$ in $\m_I$ until later.
From the hypothesis it follows that $\m$ is a
Carleson measure, i.e. for each $f\in H^2$
$$
\int_{\D}|f(z)|^2\,d\m(z)\leq C\int_{\T}|f(e^{i\th})|^2\,d\th,
$$
with the constant $C$ not depending on $f$, and $C$ is comparable to
$$
\n{\m}_{\ast}=\sup_{J\subset\T}\frac{\m(S(J))}{|J|}.
$$
We have then
\begin{align*}
A_I&=\frac{2}{|I|^{\la}}\int_{S(I)} |F(z)-F(a_I)|^2\frac{1-|z|}{|G(z)|^2}\,dA(z)\\
&\leq C'\frac{|I|^2}{|I|^{\la}}\int_{S(I)}\left|\frac{F(z)-F(a_I)}{1-\overline{a_I}z}\right|^2\frac{1-|z|}{|G(z)|^2}\,dA(z)\\
&\leq C'|I|^{2-\la}\int_{\D}\left|\frac{F(z)-F(a_I)}{1-\overline{a_I}z}\right|^2\,d\m(z)\\
&\leq C'\n{\m}_{\ast}|I|^{2-\la}\int_{\T}\left|\frac{F(e^{i\th})-F(a_I)}{1-\overline{a_I}e^{i\th}}\right|^2\,d\th\\
&=C'\n{\m}_{\ast}(1-|a_I|)^{1-\la}\int_{\T}|F(e^{i\th})-F(a_I)|^2\frac{1-|a_I|}{|1-\overline{a_I}e^{i\th}|^2}\,d\th\\
&\leq C'\n{\m}_{\ast}\n{F}_{2,\la}^2
\end{align*}
 It remains to show that $\n{\m}_{\ast}=\n{\m_I}_{\ast}\to 0 $ as $|I|\to 0$. For arcs  $J\subset\T$ we have
 $\mu_I(S(J))=\m_I((S(J)\cap S(I))$ so we need only consider arcs $J$ that intersect $I$. Let  $J$ be  such an arc.
 If $|J|>|I|$ we have
 \begin{align*}
\frac{\m_I(S(J))}{|J|}&=\frac{\m_I((S(J)\cap S(I))}{|J|}\leq \frac{\m_I(S(I))}{|J|}\\
&\leq \frac{\m_I(S(I))}{|I|}=\frac{1}{|I|} \int_{S(I)} \frac{1-|z|}{|G(z)|^2}\,dA(z) \to 0
 \end{align*}
as $|I|\to 0$, so  $\lim_{|I|\to 0}\n{\m_I}_{\ast}=0$. If $|J|\leq |I|$ then
$J\subset 3I$ where $3I$ is the arc with same center as $I$ and length $3|I|$. Thus
\begin{align*}
\n{\m_I}_{\ast}=\sup_{J\subset\T}\frac{\m_I(S(J))}{|J|}\leq
\sup_{J\subset 3I}\frac{1}{|J|}\int_{S(J)}\frac{1-|z|}{|G(z)|^2}\,dA(z).
\end{align*}
But if  $|I|\to 0$ then $|J|\leq 3|I|\to 0$ and then $\frac{1}{|J|}\int_{S(J)}\frac{1-|z|}{|G(z)|^2}\,dA(z)\to 0$. It follows that
$\n{\m_I}_{\ast} \to 0$ as $|I|\to 0$, therefore $\lim_{|I|\to 0}A_I=0$ and the proof is finished.
\end{proof}

As a corollary we obtain an analogue of \cite[Theorem 3.1]{BCDMS}.

\begin{corollary}
Let $(\phi_t)$ be a semigroup of functions with generator $G$ and $0<\la<1$. Assume that for some $\a$ with $0<\a< 1/2$,
$$
\frac{(1-|z|)^{\a}}{G(z)}=\og(1), \quad \text{as}\,\, |z|\to 1.
$$
Then $[\phi_t, H^{2, \la}]=H^{2, \la}_0$.
\end{corollary}

\begin{proof}
Under the hypothesis, $\frac{(1-|z|)^{2\a}}{|G(z)|^2}\leq C<\infty$ for all $z\in
\D$ with $|z|\geq 1/2$. Then
$$
\frac{1}{|I|}\int_{S(I)} \frac{1-|z|}{|G(z)|^2}\,dA(z)
\leq C\frac{1}{|I|}\int_{S(I)} (1-|z|)^{1-2\a}\,dA(z)\leq C'|I|^{2-2\a}
$$
so that $\lim_{|I|\to 0}\frac{1}{|I|}\int_{S(I)} \frac{1-|z|}{|G(z)|^2}\,dA(z)=0$
and the conclusion follows from the previous theorem.
\end{proof}

The following is a partial converse of the above corollary for semigroups with
Denjoy-Wolff point inside the disc.

\begin{theorem}
Let $(\phi_t)$ be a semigroup with infinitesimal generator $G$ and Denjoy-Wolff
$b\in \D$. If for  some $\la\in (0,1)$ we have $[\phi_t, H^{2, \la}]=H^{2, \la}_0$,
then
$$
\lim_{|z|\to 1}\frac{(1-|z|)^{\frac{3-\la}{2}}}{G(z)}=0.
$$
\end{theorem}

\begin{proof}
Without loss of generality we may assume $b=0$. The  generator then is
$G(z)=-zP(z)$ where Re$(P)\geq 0$ on $\D$. We assume that  $P$ is nonconstant
(for $P$ constant the assertion is clear). Let
$$
\p(z)=\int_0^z\frac{\z}{G(\z)}\,d\z=-\int_0^z\frac{1}{P(\z)}\,d\z.
$$
Since Re$(1/P)\geq 0$ the function $\p(z)$ belongs to BMOA
and thus also to $H^{2,\la}$. In addition,
$$
G(z)\p'(z)= z,
$$
a polynomial of degree $1$ which belongs to $ H^{2,\la}$. Since both $\p$ and $G\p'$ belong to
$H^{2, \la}$ it follows from  (\ref{max}) that
$\p\in [\phi_t, H^{2, \la}]$ and by the hypothesis   $\p \in H^{2, \la}_0$.

Now let $\phi_a(z)=\frac{a-z}{1-\bar{a}z}$ with $a\in \D$. Then $(\p\circ\phi_a)'(0)=\p'(a)(|a|^2-1)$ and we have
\begin{align*}
\frac{|a|(1-|a|^2)^{\frac{3-\la}{2}}}{|G(a)|}&=
(1-|a|^2)^{\frac{3-\la}{2}}|\p'(a)|=(1-|a|^2)^{\frac{1-\la}{2}}|(\p\circ\phi_a)'(0)|\\
&\leq (1-|a|^2)^{\frac{1-\la}{2}}\n{\p\circ\phi_a}_{H^2}\\
&\leq (1-|a|^2)^{\frac{1-\la}{2}}|\p(a)|+ (1-|a|^2)^{\frac{1-\la}{2}}\n{\p\circ\phi_a-\p(a)}_{H^2}.
\end{align*}
But since $\p\in H^{2, \la}_0$, we have
$\lim_{|a|\to 1}(1-|a|^2)^{\frac{1-\la}{2}}\n{\p\circ\phi_a-\p(a)}_{H^2}=0$ and
also $\lim_{|a|\to 1}(1-|a|^2)^{\frac{1-\la}{2}}|\p(a)|=0 $, so the conclusion follows.
\end{proof}

\subsection{Strong continuity on the whole space} As we have mentioned earlier, for every  nontrivial
semigroup $(\phi_t)$ the maximal space of strong continuity $[\phi_t, X]$ is
a proper subspace of $X$ when $X=H^{\infty}$ or $X=\cb$,  the Bloch space. A proof of this uses the fact
that each space is a Grothendieck
space and has the Dunford-Pettis property. H. Lotz \cite{Lotz} has proved that if $X$ is a Banach space with these two
properties then  every strongly continuous operator semigroup is automatically continuous in the uniform operator topology of
$X$. This means that the infinitesimal generator is a bounded operator on $X$. In the case of composition semigroups the
infinitesimal generator is a differential operator, and as such it is not bounded on $H^{\infty}$ or on $\cb$ unless it is the zero
operator. The same phenomenon appears on   all generalized Bloch spaces $\cb^{\a}$, $\a>0$, which are defined by
$$
\cb^{\a}=\{f: \sup_{z \in \D}(1-|z|^2)^{\a}|f'(z)|<\infty\},
$$
with norm $\n{f}_{\cb^{\a}}=|f(0)|+\sup_{z \in \D}(1-|z|^2)^{\a}|f'(z)|$. Each such space is a Grothendieck spaces with
the Dunford-Pettis property see \cite{BCDMPS}, so the theorem of Lotz applies.

The space BMOA does not have the geometric Banach space properties mentioned above. Nevertheless $[\phi_t, BMOA]$ is also
a proper subspace of BMOA for every nontrivial $(\phi_t)$. This was proved in \cite{AndersonJovovicSmith} in a more general
form:\\

\textbf{Theorem}(\cite[Theorem 1.1]{AndersonJovovicSmith}) \textit{Let  $X$ be a Banach space of analytic functions on $\D$ such that
$H^{\infty}\subseteq X\subseteq \cb$. Then for every  nontrivial $(\phi_t)$, $[\phi_t, X]\subsetneq X$}.\\

Note that this theorem gives a new function theoretic proof for $H^{\infty}$ and $\cb$. The proof
 consists of finding  a suitable test function $f$ in $
H^{\infty}$ (interpolating Blaschke product with zeros along a radius) such that $\limsup_{r\to 1^{-}}|f'(r)|(1-r)>0$ and $f$
extends to be analytic at a neighborhood of all boundary points  $\z\in \T\setminus\{1\}$. This then is used, together with the
boundary behavior of the associated univalent function of $(\phi_t)$, analyzed with the aid of the theory of prime ends, to show
that the Bloch norm of $f\circ\phi_t-f$ stays away from zero as  $t$ approaches $0$. These arguments can be adapted to the
case of Morrey spaces, with different test function, to prove an analogous result.

\begin{theorem}
Let $\la\in (0, 1)$ and let $X$ be  a Banach space of analytic functions on $\D$ such that $H^{2,\la}\subseteq X\subseteq
\cb^{\frac{3-\la}{2}}$. Then for every nontrivial semigroup  $(\phi_t)$ of functions we have  $ [\phi_t, X]\subsetneq X$. In
particular $ [\phi_t, H^{2, \la}]\subsetneq H^{2, \la}$.
\end{theorem}

\begin{proof}
The proof is based on the following claim (see also  Theorem 3.1 in \cite{AndersonJovovicSmith}).\\
\textbf{Claim} For every nontrivial $(\phi_t)$ there is a function $f\in H^{2,\la}$ such that
$$
\liminf_{t\to
0}\n{f\circ\phi_t-f}_{\cb^{\frac{3-\la}{2}}}\geq 1.
$$
\textit{Proof of the Claim}. Without loss of generality we may assume that the Denjoy-Wolff point $b$ is either $0$ or $1$.
Take the case $b=0$ first. Then the semigroup is
$$
\phi_t(z)=h^{-1}(e^{-ct}h(z)
$$
 where $h$ is a univalent spirallike function
mapping $\D$ onto $\O=h(\D)$ with $h(0)=0$ and Re$(c)\geq 0$.

If Re$(c)= 0$ then $\phi_t(z)$ are rotations and have the form $\phi_t(z) =e^{iat}z$ with $a\in \R-\{0\}$. Let
$$
f(z)= \frac{1}{(1-z)^{\frac{1-\la}{2}}}
$$
then $f\in H^{2, \la}$ and
$$
\lim_{r\to 1^{-}}|f'(re^{i\th})|(1-r)^{\frac{3-\la}{2}}=
\begin{cases} \frac{1-\la}{2}>0, \,\, \text{when}\, \th=0,\\
0, \qquad \quad \, \text{when} \, 0<\th< 2\pi.
\end{cases}
$$
Now for   $0<t<2\pi/|a|$ we have
\begin{align*}
\n{f\circ\phi_t-f}_{\cb^{\frac{3-\la}{2}}}&\geq \sup_{0<r<1}|f'(re^{iat})e^{iat}-f'(r)|(1-r)^{\frac{3-\la}{2}}\\
&\geq \limsup_{r\to 1^{-}}|f'(re^{iat})e^{iat}-f'(r)|(1-r)^{\frac{3-\la}{2}}\\
&\geq \frac{1-\la}{2}.
\end{align*}

 Consider now the case  Re$(c)>0$. As in   Theorem 3.1 of \cite{AndersonJovovicSmith} we can choose a point
$w\in
\partial\O$ such that
$$
|w|=\text{dist}(0,  \partial\O).
$$
Then $[0, w)\subset \O$ and we can view the segment $[0, w)$ as a curve $\g(s)=sw, 0\leq s<1$, ending at the point $w\in
\partial\O$. Then the inverse image $\Gamma(s)=h^{-1}(\g(s))$ is a curve in $\D$ ending at some point $\z\in \T$
\cite[Proposition 2.14]{Pommerenke}. Thus $h$ has the limit $w$ along the curve $\Gamma$ and by \cite[Corollary
2.17]{Pommerenke} it has  radial limit $w$ at $\z$, i.e. $\lim_{r\to 1^{-}}h(r \z)=w$. For each $t>0$ then,
$$
\lim_{r\to 1^{-}}\phi_t(r\z)=\lim_{r\to 1^{-}}h^{-1}(e^{-ct}h(r\z))=
h^{-1}(e^{-ct}w)\in \D.
$$

Since $\phi_t$ is a bounded univalent function, $\lim_{r\to 1^{-}}|\phi_t'(r\z)|(1-r)=0$ so also $ \lim_{r\to
1^{-}}|\phi_t'(r\z)|(1-r)^{\frac{3-\la}{2}}=0$. Letting $f_{\z}(z)=f(\overline{\z}z)$ we have
$$
\lim_{r\to 1^{-}}|f_{\z}'(r\z)|(1-r)^{\frac{3-\la}{2}}=
\lim_{r\to 1^{-}}|f'(r)|(1-r)^{\frac{3-\la}{2}}=\frac{1-\la}{2}
$$
In addition $f_{\z}'$ is continuous on $\D$ so for fixed $t>0$,
$$
\lim_{r\to 1^{-}}|f_{\z}'(\phi_t(r\z))|=|f_{\z}'(h^{-1}(e^{-ct}w))|<\infty.
$$
Thus for all $t>0$,
\begin{equation}\label{nst}
\n{f_{\z}\circ\phi_t-f_{\z}}_{\cb^{\frac{3-\la}{2}}}\geq \limsup_{r\to
1^{-}}|f_{\z}'(\phi_t(r\z))\phi_t'(r\z)
-f_{\z}(r\z)|(1-r)^{\frac{3-\la}{2}}\geq\frac{1-\la}{2}.
\end{equation}

Thus in both cases $\n{f\circ\phi_t-f}_{\cb^{\frac{3-\la}{2}}}\geq\frac{1-\la}{2}>0$ for all $t$ close to $0$. Multiplying the
test function in each case by $2/(1-\la)$ gives the result.

If the Denjoy-Wolff point $b$ is the point $1$ then
$$
\phi_t(z)=h^{-1}(h(z)+ct)
$$
where $h$ is a univalent function mapping $\D$ onto a close-to-convex domain $\O$ with $h(0)=0$ and Re$(c)> 0$. If
$(\phi_t)$ consists of automorphisms of the disc then the map $z\to z+ct$ is an automorphism of $\O$ for each $t$, so $\O$ is
either a half-plane or a strip, and the finite part of $\partial\O$ is either a line or two parallel lines. Take any point $w$ in the
finite part of $\partial\O$, then there is a point $\z\in \T$ such that $h(\z)=w$. Then $\phi_t(\z)\in\partial\D$ and for each
$t>0$, $\phi_t(\z)\ne \z$. Let $f_{\z}(z)=(1-\overline{\z}z)^{-\frac{1-\la}{2}}$. Then for $\b\in \T$,
$$
\lim_{r\to 1^{-}}|f_{\z}'(r\b)|(1-r)^{\frac{3-\la}{2}}=\begin{cases}\frac{1-\la}{2}, \,\, \text{when}\, \b=\z,\\
0, \quad\,\,  \text{when}\, \b\ne\z,
\end{cases}
$$
In addition $f_{\z}'$ extends continuously on $\overline{\D}\setminus\{\z\}$, and in particular at each  point
$\z_t=\phi_t(\z)\in \T$, $t>0$. Note that $\z_t\ne \z$. Thus
$$
\lim_{r\to 1^{-}}f_{\z}'(\phi_t(r\z))=f_{\z}'(\z_t).
$$
Also $\phi_t'$ is a bounded function on $\D$ for each $t>0$. For fixed $t$ then
$$
\n{f_{\z}\circ\phi_t-f_{\z}}_{\cb^{\frac{3-\la}{2}}}\geq \limsup_{r\to
1^{-}}|f_{\z}'(\phi_t(r\z))\phi_t'(r\z)
-f_{\z}(r\z)|(1-r)^{\frac{3-\la}{2}}\geq\frac{1-\la}{2},
$$
and replacing $f_{\z}$ by $\frac{2}{1-\la}f_{\z}$ we have the assertion in the case when $\phi_t$ are automorphisms.

Finally if $\phi_t$ are not automorphisms, the map $z\to z+ct$ is not onto $\O$. Following the arguments of the analogous case
in  Theorem 3.1 of \cite{AndersonJovovicSmith}, let $t>0$ and pick a $z\in \O\setminus (\O+ct)$. Then there is $w\in
\partial\O$ and $t_0\in (0, t]$ such that $w+ct_0=z$ and $(w, z]\subset\O$. There is a $\z\in\T$ such that
$\lim_{r\to 1^{-}}h(r\z)=w$, and the argument for the case of interior Denjoy-Wolff point repeats word-for-word to obtain a
positive lower bound for the norm $\n{f_{\z}\circ\phi_t-f_{\z}}_{\cb^{\frac{3-\la}{2}}}$ as in (\ref{nst}). Again replacing
$f_{\z}$ by $\frac{2}{1-\la}f_{\z}$ gives the  assertion of the claim.

Having proved the claim we continue to finish the proof of the theorem. The test functions we used above are in $H^{2, \la}$ so
also in $X$. The identity embedding  map $i : X \to \cb^{\frac{3-\la}{2}}$ has closed graph so by the Closed Graph Theorem it
is a bounded operator, which means that there is a constant $C$ such that $\n{f}_{\cb^{\frac{3-\la}{2}}}\leq C\n{f}_X$ for
each $f\in X$. In particular then the test functions that we have used are not in $[\phi_t, X]$ so $[\phi_t, X]\subsetneq X$.
\end{proof}

\end{document}